\documentclass[10pt,twoside,english,reqno,a4paper]{amsart}

\usepackage{listings,graphicx,amsmath,varioref,amscd,amssymb,color,bm,stmaryrd,amsthm,amsfonts,graphics,latexsym,pgf,pst-all} 
\theoremstyle{plain}
\usepackage{esint}

\usepackage{enumerate}

\theoremstyle{plain}
\newtheorem{theorem}{Theorem}[section]

\newtheorem{lemma}[theorem]{Lemma}

\newtheorem{corollary}[theorem]{Corollary}

\usepackage{geometry}
\usepackage[colorlinks=false]{hyperref}

\theoremstyle{definition}
\newtheorem{definition}[theorem]{Definition}

\newtheorem{remark}[theorem]{Remark}
\newtheorem{notation}[theorem]{Notation}

\theoremstyle{remark}

\usepackage{fouriernc}
\usepackage[T1]{fontenc}

\numberwithin{equation}{section}

\newcommand{\na}{\mathbb{N}}

\newcommand{\N}{\mathbb{N}}
\newcommand{\R}{\mathbb{R}}

\def\sob{W^{1,p}_{0}(\Omega)}

\def\lio{{L^{\infty}(\Omega)}}

\def\div{\mathrm{div}}
\def\into{\int_{\Omega}}

\def\ae{\mathrm{a.e.}}

\newcommand{\oF}{\overline{F}}
\newcommand{\uF}{\underline{F}}
\newcommand{\oL}{L^+}
\newcommand{\uL}{L^-}
\newcommand{\oLL}{L^+_\Lambda}
\newcommand{\uLL}{L^-_\Lambda}

\begin{document}
\title[Existence and nonexistence of infinitely many solutions to elliptic problems]{Existence and nonexistence of infinitely many solutions\\
to elliptic problems with oscillating nonlinearities}

\author[A. J. Mart\'{i}nez Aparicio]{Antonio J. Mart\'{i}nez Aparicio}
\address[Antonio J. Mart\'{i}nez Aparicio]{Departamento de Matem\'aticas\newline Uni\-ver\-si\-dad de Alme\-r\'ia\newline Ctra. Sacramento s/n\newline
La Ca\~{n}ada de San Urbano\newline 04120 - Al\-me\-r\'{\i}a, Spain}
\email{ajmaparicio@ual.es}

\author[C. Torres-Latorre]{Clara Torres-Latorre}
\address[Clara Torres-Latorre]{Instituto de Ciencias Matemáticas\newline Consejo Superior de Investigaciones Científicas\newline
	C/ Nicolás Cabrera, 13-15\newline 28049 - Madrid, Spain.}
\email{\tt clara.torres@icmat.es}

\keywords{Infinitely many solutions, bifurcation, quasilinear equation, $p$-Laplacian, non-divergence form}
  
\subjclass[2020]{35J15, 35J62, 35B09, 35B32}

\begin{abstract}
We study sharp conditions for the existence and nonexistence of infinitely many
nonnegative solutions to the problem \mbox{$-\Delta_p u = \lambda f(u)$} in a bounded
domain with Dirichlet boundary conditions, where $f$ is a continuous
function with a sequence of positive zeros converging to zero or
diverging to infinity. Under a growth condition on the primitive
$F(s) = \int_0^s f(t)\,\mathrm{d}t$, we establish ranges of
the parameter $\lambda$ for which infinitely many small or large
solutions exist, as well as ranges where no bifurcation from zero
or infinity can occur. The existence result is obtained via
variational methods for a general class of divergence form operators, while the nonexistence result is established both for the $p$-Laplacian and for uniformly elliptic operators in non-divergence form via an ODE argument.
\end{abstract}

\maketitle

\section{Introduction}

In this article, we investigate the existence and nonexistence of infinitely many nonnegative solutions to problem
\begin{equation}
	\label{eq:Pb_pLap}
	\begin{cases}
		-\Delta_p u = \lambda f(u) & \text{in } \Omega,\\
		u=0 & \text{on } \partial\Omega.
	\end{cases}
\end{equation}
Here, $\Omega\subset \R^N$ ($N\geq 1$) is a bounded domain, $\Delta_p u = \operatorname{div}(|\nabla u|^{p-2} \nabla u)$ denotes the $p$-Laplacian operator ($p>1$), and $\lambda>0$ is a constant. The nonlinearity $f\colon \R \to \R$ is a continuous function, possibly changing sign.

We address, with a unified approach, two different issues: the existence of infinitely many small solutions and the existence of infinitely many large solutions. By \textit{infinitely many small solutions}, we refer to a sequence $\{u_n\}$ of nonnegative (and nontrivial) solutions to~\eqref{eq:Pb_pLap} such that $\|u_n\|_\lio \to 0$. Conversely, we say that~\eqref{eq:Pb_pLap} admits \textit{infinitely many large solutions} if there exists a sequence of bounded solutions $\{u_n\}$ such that $\|u_n\|_\lio \to \infty$.

We establish precise ranges of the parameter $\lambda$ for which infinitely many solutions to~\eqref{eq:Pb_pLap} exist, as well as 
ranges where they do not, under a suitable oscillatory behaviour of $f$. Our existence result significantly generalizes the preceding works~\cite{AC02, CMM25, OZ96}, and to our knowledge, the nonexistence result is the first of its kind in this generality. Moreover, the techniques developed herein can be readily extended to a broader class of operators and nonlinearities.

In what follows, we present our main results for the model problem~\eqref{eq:Pb_pLap}. Subsequently, we provide a brief contextualization within the existing literature and discuss the current state of the art.

\subsection{Main results} 

Our analysis rests upon a single, fundamental assumption regarding the structure of the nonlinearity $f$: for $\ell\in \{0,\infty\}$, 
\begin{equation}
\label{eq:hyp_zeros}
\text{there exists $\{\alpha_n\}\subset \R^+$ such that $\alpha_n\to \ell$ and $f(\alpha_n)=0$ for any $n\in\N$.}
\end{equation}
Under this assumption, we show that the existence and nonexistence of infinitely many solutions to~\eqref{eq:Pb_pLap} depends solely on the asymptotic behaviour of
\[
F(s):=\int_0^s f(t)\, \mathrm{d}t,\ \forall s\geq 0,
\]
at the origin (if $\ell=0$) or at infinity (if $\ell = \infty$).

Our first result concerns the existence of infinitely many solutions.

\begin{theorem}[Infinitely many solutions]
\label{th:Intro_Exist}
Let $\ell\in \{0,\infty\}$, and let $f\in C(\R)$ be a function with $f(0)\geq 0$. Suppose that~\eqref{eq:hyp_zeros} holds, and that
\begin{equation}
\label{eq:hyp_infsup}
    -\infty <\liminf_{s\to \ell^+} \frac{F(s)}{s^p} \leq \limsup_{s\to \ell^+} \frac{F(s)}{s^p} = \oL \in (0,\infty].
\end{equation}
Then, there exists $\overline{\lambda} \in [0,\infty)$ such that, for any $\lambda>\overline{\lambda}$, problem~\eqref{eq:Pb_pLap} has a sequence $\{u_n\}$ of nonnegative (and nontrivial) bounded solutions such that $\|u_n\|_\lio \to \ell$.

Furthermore, if $\oL=\infty$, then $\overline{\lambda}=0$.
\end{theorem}

\begin{notation}
The limits in Theorem \ref{th:Intro_Exist} are taken as $s\to0^+$ if $\ell=0$ and as $s\to\infty$ if $\ell=\infty$.
\end{notation}

Actually, Theorem~\ref{th:Intro_Exist} is obtained as a direct consequence of a more general result, Theorem~\ref{th:General_Exist}, where the growth condition~\eqref{eq:hyp_infsup} is significantly relaxed. The proof of Theorem~\ref{th:General_Exist} is conducted using variational methods combined with a suitable truncation argument.

Regarding the nonexistence of infinitely many small or large solutions, we establish a stronger result concerning the absence of bifurcation points. We recall that a value $\tilde \lambda\geq 0$ is said to be a \textit{bifurcation point from $\ell\in\{0,\infty\}$} for problem~\eqref{eq:Pb_pLap} if there exists a sequence of pairs $\{(\lambda_n,u_n)\} \subset \R^+ \times L^\infty(\Omega)$ with $u_n\geq 0$ and $u_n \not\equiv 0$ solving~\eqref{eq:Pb_pLap} such that $\lambda_n\to \tilde\lambda$ and $\|u_n\|_\lio \to \ell$.

Below, we state the result regarding the nonexistence of bifurcation points.

\begin{theorem}[Nonexistence of bifurcation points]
\label{th:Intro_NonExist}
Let $\ell\in \{0,\infty\}$, and let $f\in C(\R)$. Suppose that~\eqref{eq:hyp_zeros} holds and that
\begin{equation}
\label{eq:hyp_infsup2}
    -\infty <\uL = \liminf_{s\to \ell^+} \frac{F(s)}{s^p} \leq \limsup_{s\to \ell^+} \frac{F(s)}{s^p} = \oL < \infty.
\end{equation}
Then, there exists $\underline{\lambda}\in (0,\infty]$ such that no $\lambda\in [0, \underline{\lambda})$ is a bifurcation point from $\ell$ for problem~\eqref{eq:Pb_pLap}.

Furthermore, if $\oL<0$ or $\uL=\oL= 0$, then $\underline{\lambda}=\infty$.
\end{theorem}

To show this result, we combine the sub-supersolution method with a symmetry argument. This approach allows us to reduce the analysis to a study of the resulting ordinary differential equation.

Note that Theorems~\ref{th:Intro_Exist} and~\ref{th:Intro_NonExist} are complementary and can be combined to obtain a full picture of problem~\eqref{eq:Pb_pLap}. Together, they provide sharp conditions on the parameter $\lambda$ that separate the regions of infinite multiplicity from those where bifurcation is impossible. We illustrate these combined results in the case where $F\geq 0$.

\begin{corollary}
\label{cor:F_nonnegative}
Let $\ell\in \{0,\infty\}$, and let $f\in C(\R)$ be such that $F\geq 0$. Suppose that~\eqref{eq:hyp_zeros} holds and denote
\begin{equation*}
    \limsup_{s\to \ell^+} \frac{F(s)}{s^p} = \oL \in [0,\infty].
\end{equation*}
The following statements hold:
\begin{enumerate}[i)]
    \item If $\oL=0$, then problem~\eqref{eq:Pb_pLap} has no bifurcation point from $\ell$.
    \item If $\oL\in (0,\infty)$, then there exist constants $0 < \underline{\lambda} \leq \overline{\lambda} < \infty$ such that:
    \begin{itemize}
        \item No $\lambda\in [0,\underline{\lambda})$ is a bifurcation point from $\ell$.
        \item For any $\lambda\in (\overline{\lambda},\infty)$, problem~\eqref{eq:Pb_pLap} admits infinitely many small (if $\ell=0$) or large (if $\ell=\infty$) solutions.
    \end{itemize}
    \item If $\oL=\infty$, then problem~\eqref{eq:Pb_pLap} admits infinitely many small (if $\ell=0$) or large (if $\ell=\infty$) solutions for any $\lambda>0$.
\end{enumerate}

\end{corollary}

\subsection{Background}

The study of the existence of infinitely many solutions for problems such as~\eqref{eq:Pb_pLap} has been a highly active line of research in recent decades. It is now well-established that this phenomenon is triggered by the oscillatory behaviour of the nonlinearity $f$ or, more precisely, of its primitive $F$.

In~\cite{OZ96}, Omari and Zanolin investigated problem~\eqref{eq:Pb_pLap} under the assumptions $f(0)\geq 0$ and
\begin{equation}
\label{eq:Intro_OZ}
-\infty<\liminf_{s\to \infty} \frac{F(s)}{s^p} \leq 0 \qquad \text{and} \qquad \limsup_{s\to \infty} \frac{F(s)}{s^p} = \infty.
\end{equation}
Employing a strategy based on time-mapping estimates, they proved that problem~\eqref{eq:Pb_pLap} has an unbounded sequence of nonnegative solutions for any $\lambda>0$. Since then, the approach of considering highly oscillating nonlinearities---where the limit of $\frac{F(s)}{s^p}$ at zero or infinity does not exist---has been further explored. In these subsequent works, the range of $\lambda$ for which infinite solutions exist is determined by the limit inferior and limit superior. For instance, in~\cite{BMR12} the authors showed for a nonnegative $f$ that condition
\begin{equation}
\label{eq:Intro_OZ2}
0\leq \liminf_{s\to 0^+} \frac{F(s)}{s^p}  < \limsup_{s\to 0^+} \frac{F(s)}{s^p} = \infty,
\end{equation}
guarantees the existence of infinitely many small solutions to~\eqref{eq:Pb_pLap} for any $\lambda\in \big(0, \big(C\liminf_{s\to 0^+} \frac{F(s)}{s^p}\big)^{-1}\big)$, with $C>0$. Other notable contributions in this direction are~\cite{BM10, OO06}.

A different approach, based on a variational principle of Ricceri (\cite{Ric00}), was introduced by Anello and Cordaro in~\cite{AC02}. In their setting, the oscillating behaviour of $f$ is characterized by the following assumption: there exist two sequences $0<s_n<s_n'$ with $s_n'\to 0$ such that
\begin{equation}
\label{eq:Intro_AC1}
F(s_n) = \max_{s\in[s_n,s_n']} F(s).
\end{equation}
Under condition~\eqref{eq:Intro_AC1}, the authors established the existence of infinitely many small solutions to~\eqref{eq:Pb_pLap} for any $\lambda>0$, provided the nonlinearity simply satisfies
\begin{equation*}
-\infty < \liminf_{s\to 0^+} \frac{F(s)}{s^p}  \leq \limsup_{s\to 0^+} \frac{F(s)}{s^p} = \infty.
\end{equation*}
For nonlinearities with an oscillatory behaviour at infinity, a similar strategy remains effective provided that $f$ has a subcritical growth (see~\cite{MP14}). It is worth noting that, within this framework, the limit of $\frac{F(s)}{s^p}$ at zero or infinity may exist. The approach introduced in~\cite{AC02} has had a significant impact on the field, and its results have been generalized and extended to a wide variety of contexts (see, for instance,~\cite{HZ09, MP14}).

It should be observed that~\eqref{eq:Intro_AC1} forces $f$ to be nonpositive within the non-degenerate intervals $[s_n,s_n']$. This assumption is remarkably restrictive, as it excludes the case where $f\geq 0$ has only isolated zeros. Such nonlinearities are studied in~\cite{CMM25} with the usual Laplacian, where the authors consider functions of the form $f(s)=s^{1/r}\big(1+\sin\frac{1}{s}\big)$ or $f(s)=s^r(1+\sin s)$, with $r>0$. Notably, these functions satisfy neither~\eqref{eq:Intro_OZ} nor~\eqref{eq:Intro_OZ2}, since the limit of $\frac{F(s)}{s^p}$ at zero or infinity actually exists. By constructing explicit subsolutions, they show the existence of infinitely many solutions for any $\lambda>0$ when $r>1$, and for $\lambda>\overline{\lambda}>0$ when $r=1$. Further results on the existence of infinitely many solutions can be found in~\cite{CCGS16, FS25, KM10,  Ryn00}.

In the present work, we introduce a novel approach that significantly generalizes the framework established in~\cite{AC02} and encompasses the classes of nonlinearities investigated in~\cite{CMM25}. Our central assumption is~\eqref{eq:hyp_zeros}: we require that $f$ has a sequence of positive zeros that either converges to zero or diverges to infinity. Unlike~\eqref{eq:Intro_AC1}, condition~\eqref{eq:hyp_zeros} covers the case where $f\geq 0$ has a sequence of isolated zeros. Furthermore, while there exist nonlinearities satisfying~\eqref{eq:Intro_OZ} or~\eqref{eq:Intro_OZ2} but not~\eqref{eq:hyp_zeros}, most of the practical examples addressed in~\cite{OZ96} and subsequent related studies involve nonlinearities with large dead zones; that is, they satisfy condition~\eqref{eq:hyp_zeros} (see, for instance,~\cite[Section~4]{BM10} or~\cite[Section~4]{BMR12}).

A general example covered by our assumptions but not by those of preceding works is the nonlinearity 
\[
f(s)=g(s)(1+\sin s),\ \forall s\geq 0,
\]
where $g\colon \R\to \R^+$ is continuous and nondecreasing. In this case, if $\lim_{s\to \infty} \frac{\int_0^s g(t)\, \mathrm{d}t}{s^p} = \oL \in [0,\infty]$, then it follows that $\lim_{s\to \infty} \frac{F(s)}{s^p} = \oL$ as well, and Corollary~\ref{cor:F_nonnegative} directly applies.

\subsection{Plan of the paper}

The strategies developed in this work are robust; indeed, the theorems stated above can be extended to other operators. To illustrate this, we show that Theorem~\ref{th:Intro_Exist} remains valid for problems of the form
\begin{equation}
	\label{eq:div_problem}
	\begin{cases}
		-\div(\mathbf{a}(x,\nabla u)) = \lambda f(u) & \text{in } \Omega,\\
		u=0 & \text{on } \partial\Omega,
	\end{cases}
\end{equation}
where $\div(\mathbf{a}(x,\nabla u))$ is a variational operator with $p$-growth that generalizes the $p$-Laplacian. Furthermore, a modified version of Theorem~\ref{th:Intro_NonExist} holds for the non-divergence form problem
\begin{equation}
    \label{eq:nondiv_problem}
    \begin{cases}
    -\mathcal{L}u = \lambda f(u) & \text{in } \Omega, \\
    u = 0 & \text{on } \partial\Omega,
    \end{cases}
\end{equation}
where $\mathcal{L}u= \operatorname{Tr}(AD^2u)$ is a linear elliptic operator with bounded measurable coefficients.

The rest of the paper is organized as follows. In Section~\ref{sec:setting}, after establishing some notation, we specify the exact assumptions on the operators in problems~\eqref{eq:div_problem} and~\eqref{eq:nondiv_problem}. Next, in Section~\ref{sec:existence}, a more general version of Theorem~\ref{th:Intro_Exist} is proved for problem~\eqref{eq:div_problem}. Finally, Section~\ref{sec:nonexistence} is devoted to the proof of the nonexistence Theorem~\ref{th:Intro_NonExist}, along with a suitably adapted version for problem~\eqref{eq:nondiv_problem}.

\section{Setting}
\label{sec:setting}

Throughout this paper, $\Omega \subset \R^N$ denotes a bounded open domain, and $f : \R \to \R$ is a continuous function. For each $s\geq 0$, we write 
\[
F(s) = \int_0^sf(t)\, \mathrm{d}t.
\]
We also define, for all $s\geq 0$, the auxiliary function
\begin{equation}
    \label{eq:def_ov_F}
    \oF(s) := \max \left\{ \int_t^{s} f(\rho)\, \mathrm{d}\rho : 0\leq t\leq s \right\} = F(s) - \min_{t\in[0,s]} F(t).
\end{equation}
In particular, if $f\geq 0$, then $\oF = F$.

When dealing with uniformly elliptic non-divergence form operators, we use a modified version of $F$. For $s\geq 0$, we denote
\begin{equation}
    \label{eq:def_F_Lambda}
    F_\Lambda(s) := \int_0^s f^+(t)\,\mathrm{d}t 
    - \frac{1}{\Lambda^2}\int_0^s f^-(t)\,\mathrm{d}t,
\end{equation}
where $\Lambda \geq 1$ is the ellipticity constant, $f^+ := \max\{f,0\}$ and $f^- := \max\{-f,0\}$. Analogously to $\oF$, for each $s\geq 0$ we define the function
\begin{equation}
    \label{eq:def_ov_F_Lambda}
    \overline{F}_\Lambda(s) := \max \left\{ \int_t^s f^+(\rho) \, \mathrm{d}\rho - \frac{1}{\Lambda^2}\int_t^s f^-(\rho) \, \mathrm{d}\rho : 0 \leq t \leq s \right\} = F_\Lambda(s) - \min_{t\in[0,s]} F_\Lambda(t).
\end{equation}
Note that $F_1=F$ and $\oF_1 = \oF$.

\subsection{Divergence form operators}
We consider operators of the form 
\begin{equation*}
    \mathcal{L}u = \div(\mathbf{a}(x,\nabla u)),
\end{equation*}
where $\mathbf{a} : \Omega \times \R^N \to \R^N$ is a Carath\'eodory function (i.e. measurable in $x$ for all fixed $\xi$ and continuous in $\xi$ for almost every fixed $x$) satisfying:

\begin{itemize}
\item[\hypertarget{hyp:potential}{(H1)}] (Potential structure) $\mathbf{a}(x,\xi) = \nabla_\xi \Phi(x,\xi)$ for a function $\Phi:\Omega\times\R^N\to\R$ with $\Phi(x,0)=0$ for $\ae$ $x\in\Omega$.

\item[\hypertarget{hyp:growth}{(H2)}] (Growth and coercivity) There exist constants $0 < \alpha \leq \beta$ and $1 < p < \infty$ such that for $\ae$ $x \in \Omega$ and all $\xi \in \R^N$,
\begin{equation*}
\frac{\alpha}{p} |\xi|^p \leq \Phi(x,\xi) \leq \frac{\beta}{p} |\xi|^p.
\end{equation*}

\item[\hypertarget{hyp:convex}{(H3)}] (Strict convexity) The map $\xi \mapsto \Phi(x,\xi)$ is strictly convex for $\ae$ $x \in \Omega$.
\end{itemize}

Conditions \textnormal{(\hyperlink{hyp:potential}{H1})--(\hyperlink{hyp:convex}{H3})} imply the standard growth, coercivity, and monotonicity conditions on $\mathbf{a}$. Indeed, since $\Phi$ is convex with $\Phi(x,0)=0$, one has $\Phi(x,\xi) \leq \mathbf{a}(x,\xi)\cdot\xi$, which together with the lower bound in \textnormal{({\hyperlink{hyp:growth}{H2}})} gives the coercivity $\mathbf{a}(x,\xi)\cdot\xi \geq \frac{\alpha}{p}|\xi|^p$. For the growth bound, convexity yields $\mathbf{a}(x,\xi)\cdot\eta \leq \Phi(x,\xi+\eta)$ for all $\eta$; taking $\eta= \mathbf{a}(x,\xi) |\mathbf{a}(x,\xi)|^{-1} |\xi|$ and using the upper bound in \textnormal{({\hyperlink{hyp:growth}{H2}})} gives $|\mathbf{a}(x,\xi)| \leq \frac{\beta}{p}\, 2^p|\xi|^{p-1}$. Finally, strict convexity \textnormal{(\hyperlink{hyp:convex}{H3})} implies strict monotonicity: $(\mathbf{a}(x,\xi)-\mathbf{a}(x,\eta))\cdot(\xi-\eta)>0$ for $\xi\neq\eta$.

This framework includes the $p$-Laplacian (where $\Phi(x,\xi) = \frac{1}{p}|\xi|^p$ and $\textbf{a}(x,\xi) = |\xi|^{p-2}\xi$) as well as uniformly elliptic divergence form operators in the linear case $p=2$, given by $\Phi(x,\xi) = \frac{1}{2}\xi^T A(x) \xi$ and $\textbf{a}(x,\xi) = A(x) \xi$, with $A$ symmetric and $\alpha I \leq A(x) \leq \beta I$.

\subsection{Non-divergence form operators}
We also consider non-divergence form elliptic operators with bounded measurable coefficients,
\begin{equation}\label{eq:nondiv_op}
    \mathcal{L}u = \sum\limits_{i,j = 1}^Na_{ij}(x)\partial^2_{ij}u, \quad \frac{1}{\Lambda} I \leq A(x) \leq \Lambda I,
\end{equation}
where $\Lambda \geq 1$. The extremal operators among this class are the Pucci operators, which play a key role in the theory of fully nonlinear elliptic equations.
\begin{definition}
The \emph{Pucci extremal operators} are defined by
\begin{align*}
\mathcal{M}^+(D^2u) &:= \sup\limits_{\frac{1}{\Lambda} I \leq M \leq \Lambda I}\operatorname{Tr}(MD^2u) = \Lambda \sum_{e_i > 0} e_i + \frac{1}{\Lambda} \sum_{e_i < 0} e_i, \\
\mathcal{M}^-(D^2u) &:= \inf\limits_{\frac{1}{\Lambda} I \leq M \leq \Lambda I}\operatorname{Tr}(MD^2u) = \frac{1}{\Lambda} \sum_{e_i > 0} e_i + \Lambda \sum_{e_i < 0} e_i,
\end{align*}
where $e_i$ are the eigenvalues of $D^2u$.
\end{definition}

For any uniformly elliptic operator $\mathcal{L}$ satisfying \eqref{eq:nondiv_op}, we have $\mathcal{M}^-(D^2u) \leq \mathcal{L}u \leq \mathcal{M}^+(D^2u)$. See \cite{CC95} or \cite{FR22} for further properties.

\section{Existence of infinitely many solutions}
\label{sec:existence}

This section is devoted to the proof of Theorem~\ref{th:Intro_Exist}. In fact, we establish the following more general version of it, which relaxes the growth conditions on the nonlinearity and extends the result to the broader class of operators defined in~\eqref{eq:div_problem}.

\begin{theorem}
\label{th:General_Exist}
Let $\ell\in \{0,\infty\}$, and let $f\in C(\R)$ be a function with $f(0)\geq 0$. Assume~\textnormal{(\hyperlink{hyp:potential}{H1})--(\hyperlink{hyp:convex}{H3})}. Suppose that~\eqref{eq:hyp_zeros} holds and that there exists $\{\gamma_n\}\subset \R^+$ with $\gamma_n \to \ell$ such that
    \begin{equation}
    \label{eq:hyp_general}
        \frac{F(\gamma_n)}{\gamma_n^p} \to \oL\in (0,\infty] \qquad \text{and} \qquad  -\min_{t\in [0,\gamma_n]} F(t) \leq M F(\gamma_n) \text{ for some } M \geq 0.
    \end{equation}
Then, there exists $\overline{\lambda} \in [0,\infty)$ such that, for any $\lambda>\overline{\lambda}$, problem~\eqref{eq:div_problem} has a sequence $\{u_n\}$ of nonnegative (and nontrivial) bounded solutions such that $\|u_n\|_\lio \to \ell$.

Furthermore, if $\oL=\infty$, then $\overline{\lambda}=0$.
\end{theorem}

\begin{remark}
Observe that $f(0)=0$ when $\ell=0$. In this case, $u\equiv 0$ is always a solution to~\eqref{eq:div_problem}.
\end{remark}

\begin{remark}
There exist functions satisfying~\eqref{eq:hyp_general} but not~\eqref{eq:hyp_infsup}. An example is $F(s)=s^{p+1}\cos s$, for which $\liminf_{s\to \infty}\frac{F(s)}{s^p}=-\infty$ but~\eqref{eq:hyp_general} holds by choosing $M=1$ and $\gamma_n=2\pi n$.
\end{remark}

Condition~\eqref{eq:hyp_general} has been widely adopted by numerous authors since it was first introduced in~\cite{AC02}. In the case $\ell=\infty$, it has usually been accompanied by a subcritical growth assumption of $f$ (see, for instance,~\cite{MP14}). The present work shows that such a growth restriction is not necessary.

On the other hand, the existence of the sequence $\{\gamma_n\}$ in~\eqref{eq:hyp_general} is always guaranteed whenever~\eqref{eq:hyp_infsup} holds, as the following lemma shows.

\begin{lemma}
Let $\ell\in \{0,\infty\}$, and let $f\in C(\R)$. Then~\eqref{eq:hyp_infsup} implies~\eqref{eq:hyp_general}.
\end{lemma}

\begin{proof}
Since $\limsup_{s\to \ell^+} \frac{F(s)}{s^p} = \oL \in (0,\infty]$, there exists a sequence $\{\gamma_n\} \subset \R^+$ with $\gamma_n \to \ell$ such that
\begin{equation}
\label{eq:Pf_equiv}
\frac{F(\gamma_n)}{\gamma_n^p} \to \oL.
\end{equation}
By~\eqref{eq:Pf_equiv}, it follows that $F(\gamma_n) > \min\{\oL/2,1\} \gamma_n^p$ when $n$ is large.

Assume first that $\ell=0$. Since $\liminf_{s\to 0^+} \frac{F(s)}{s^p} > -\infty$, there exist $M,\delta>0$ such that $F(s)\geq -Ms^p$ for any $s\in [0,\delta]$. For $n$ large, $\gamma_n$ belongs to $[0,\delta]$ and thus $\min_{t\in [0,\gamma_n]} F(t) \geq -M\gamma_n^p$. Then, for $n$ large, it holds
\[
\min_{t\in [0,\gamma_n]} F(t) \geq - M \min\{\oL/2,1\}^{-1} F(\gamma_n).
\]

Suppose now that $\ell=\infty$. We may assume that $\gamma_n\geq 1$ for every $n$. Since $\liminf_{s\to \infty} \frac{F(s)}{s^p} > -\infty$ and $F$ is bounded in compact sets, there exist $M_1,M_2>0$ such that $F(s)\geq -M_1s^p - M_2$ for any $s\geq 0$. Therefore, $\min_{t\in[0,\gamma_n]} F(t)\geq -M_1\gamma_n^p - M_2 \geq -(M_1+M_2)\gamma_n^p$ and thus we conclude that, for $n$ large,
\[
\min_{t\in[0,\gamma_n]} F(t)\geq -(M_1+M_2) \min\{\oL/2,1\}^{-1} F(\gamma_n).
\]

This ends the proof.
\end{proof}

Our approach to show Theorem~\ref{th:General_Exist} relies on a suitable truncation of the problem combined with a variational argument. For $n \in \N$, we define the truncated nonlinearity
\begin{equation*}
f_n(s) = \begin{cases}
f(0) & \text{if } s<0,\\
f(s) & \text{if } 0\leq s\leq \alpha_{n},\\
0  & \text{if } s>\alpha_{n},
\end{cases}
\end{equation*}
and we set $F_n(s) :=\int_0^s f_n(t)\, \mathrm{d}t$. Since $f(\alpha_n) = 0$, $f_n$ is continuous. 

We associate the energy functional $I_n$ to $f_n$, and show that it has a global minimizer.

\begin{lemma}\label{lem:minimizer}
Assume $f(0)\geq 0$ and \textnormal{(\hyperlink{hyp:potential}{H1})--(\hyperlink{hyp:convex}{H3})}. Then, for each $n\in\N$ and $\lambda > 0$, the functional
\[
I_n(\lambda,u) = \into \Phi(x,\nabla u) - \lambda \into F_n(u),\quad u\in\sob,
\]
attains its global minimum at some $u_n \in \sob$. Moreover, $u_n$ is a weak solution to~\eqref{eq:div_problem} with $f$ replaced by $f_n$, and $0\leq u_n\leq\alpha_{n}$.
\end{lemma}

\begin{proof}
Coercivity follows from (\hyperlink{hyp:growth}{H2}) and the boundedness of $f_n$. Indeed, using that $\Phi(x,\xi) \geq \frac{\alpha}{p}|\xi|^p$ and $|F_n(s)| \leq \|f_n\|_{L^\infty(\R)} |s|$, an application of H\"older and Sobolev inequalities yields, for some constant $C>0$, that
\[
I_n(\lambda, u) \geq \frac{\alpha}{p} \|\nabla u\|_{L^p(\Omega)}^p - \lambda C \|f_n\|_{L^\infty(\R)} \|\nabla u\|_{L^p(\Omega)}.
\]
Weak lower semicontinuity follows from the strict convexity of $\Phi$ in $\xi$. Note that given $u_m\rightharpoonup u$ in $\sob$, assumption (\hyperlink{hyp:convex}{H3}) implies that
\[
\into \Phi(x,\nabla u_m) \geq \into \Phi(x,\nabla u) + \into \nabla_\xi \Phi(x,\nabla u) \cdot (\nabla u_m - \nabla u)
\]
and then $\liminf_{m\to\infty} I_n(\lambda,u_m) \geq I_n(\lambda,u)$. Therefore, $I_n(\lambda, \cdot)$ attains its minimum at some $u_n \in \sob$.

By (\hyperlink{hyp:potential}{H1}), the minimizer satisfies the Euler--Lagrange equation $-\div(\mathbf{a}(x,\nabla u_n)) = \lambda f_n(u_n)$. Hence, $u_n$ is a weak solution to~\eqref{eq:div_problem} with $f$ replaced by $f_n$. Since $F_n(s)\leq F_n(0)$ for $s\leq 0$ and $F_n(s)$ is constant for $s > \alpha_n$, minimality ensures $0\leq u_n\leq\alpha_{n}$.
\end{proof}

The proof of Theorem~\ref{th:General_Exist} differs substantially depending on whether $\ell=0$ or $\ell=\infty$. For this reason, we divide its proof into two distinct sections: we first establish the case $\ell=0$ and then we address the case $\ell=\infty$.

\subsection{Infinitely many small solutions}

Here we establish Theorem~\ref{th:General_Exist} when $\ell=0$. Since $f(0)=0$, the trivial function $u\equiv 0$ always solves~\eqref{eq:div_problem}. In the following, we show that the solutions $u_n$ found in Lemma~\ref{lem:minimizer} are nontrivial whenever $\lambda$ is greater than some $\overline{\lambda}\geq 0$.

\begin{proof}[Proof of Theorem~\ref{th:General_Exist} (Case $\ell=0$)]
Without loss of generality, we may assume that both $\{\alpha_n\}$ and $\{\gamma_n\}$ are decreasing sequences converging to zero, with $\gamma_n \leq \alpha_n$ for all $n\in\N$. Furthermore, we can assume that $F(\gamma_n)>0$ for all $n\in\N$.

Let $\oF$ be the function defined in~\eqref{eq:def_ov_F}. Since $F\leq \oF$, passing to a subsequence if necessary, by~\eqref{eq:hyp_general} it follows that
\begin{equation}
\label{eq:Pf_small_1}
\frac{\oF(\gamma_n)}{\gamma_n^p} \to K\in (0,\infty],
\end{equation}
with $K\geq \oL$. Furthermore, by assumption~\eqref{eq:hyp_general} it also holds that
\begin{equation}
\label{eq:Pf_small_2}
\frac{F(\gamma_n)}{\oF(\gamma_n)} = \frac{F(\gamma_n)}{F(\gamma_n)- \min_{t\in[0,\gamma_n]} F(t)} \geq \frac{F(\gamma_n)}{F(\gamma_n) + M F(\gamma_n)} = \frac{1}{1+M},\quad \forall n\in\N.
\end{equation}

According to Lemma~\ref{lem:minimizer}, for each $n \in \N$, the functional $I_n(\lambda, \cdot)$ attains its global minimum at some $u_n \in \sob$ with $0 \leq u_n \leq \alpha_n$. To show that $u_n$ is nontrivial for a specific range of $\lambda$, we proceed as follows. For a sufficiently small $\delta > 0$, we denote $\Omega_{\delta} := \{ x \in \Omega : \operatorname{dist}(x, \partial \Omega) < \delta \}$ and, for each $n\in\N$, we define the functions
\begin{equation}
\label{eq:Pf_small_2bis}
w_{n}(x) := \gamma_n \min \left\{ 1, \frac{\operatorname{dist}(x, \partial \Omega)}{\delta} \right\}.
\end{equation}
Note that $w_n\in W_0^{1,p}(\Omega)$ is such that $w_n\equiv \gamma_n$ in $\Omega\setminus\Omega_\delta$ and $\|w_n\|_\lio=\gamma_n \leq \alpha_n$. Using the definition of $\oF$ and~\eqref{eq:Pf_small_2}, it follows
\begin{equation}
\label{eq:Pf_small_3}
\begin{split}
\int_{\Omega} F_n (w_n) &=
\int_{\Omega \setminus \Omega_{\delta}} F (w_n) +
\int_{\Omega_{\delta}} F(w_n) =
\int_{\Omega} F(\gamma_n) - \int_{\Omega_\delta} \left( F(\gamma_n)  - F(w_n) \right)
\\
&\geq F(\gamma_n) |\Omega| - \oF(\gamma_n) |\Omega_\delta| \geq \oF(\gamma_n) \left(\frac{1}{1+M} |\Omega| -|\Omega_\delta| \right),\quad \forall n\in\na.
\end{split}
\end{equation}
At this point, we fix $\delta>0$ sufficiently small so that $C_1:= \frac{1}{1+M} |\Omega| -|\Omega_\delta| > 0$. Taking into account the growth condition~(\hyperlink{hyp:growth}{H2}), the equality $\|\nabla w_n\|_{L^\infty(\Omega)} = \frac{\gamma_n}{\delta}$, and the inequality~\eqref{eq:Pf_small_3}, we deduce that
\[
I_n(\lambda,w_n) = \into \Phi(x,\nabla w_n) - \lambda \int_{\Omega} F_n (w_n) \leq \frac{\beta}{p} \frac{|\Omega|}{\delta^p} \gamma_n^p -\lambda C_1 \oF(\gamma_n) =: C_2 \gamma_n^p - \lambda C_1 \oF(\gamma_n),\quad \forall n\in\N.
\]
Taking $\lambda_n = \frac{C_2}{C_1} \gamma_n^p\oF(\gamma_n)^{-1}$, we obtain that $I_n(\lambda, w_n) < 0$ for all $\lambda > \lambda_n$ and hence $u_n$ is nontrivial whenever $\lambda > \lambda_n$.

Finally, we set $\overline{\lambda} := \lim \lambda_n$, which is guaranteed to be nonnegative and finite by~\eqref{eq:Pf_small_1}. Moreover, it follows from~\eqref{eq:Pf_small_1} that $\overline{\lambda}=0$ if $\oL=\infty$. For each $\lambda > \overline{\lambda}$, there is some $n_0\in\N$ such that $\lambda_n<\lambda$ for any $n\geq n_0$. Hence, $\{u_n\}_{n\geq n_0}$ is a sequence of nonnegative (and nontrivial) solutions to~\eqref{eq:div_problem} such that $\|u_n\|_\lio \leq \alpha_n$ for each $n\geq n_0$. Since $\alpha_n\to 0$, we conclude that $\|u_n\|_\lio \to 0$.
\end{proof}

\subsection{Infinitely many large solutions}

Now we establish Theorem~\ref{th:General_Exist} for the case $\ell=\infty$. In what follows, we show that the sequence of solutions $\{u_n\}$ obtained in Lemma~\ref{lem:minimizer} is unbounded in $L^\infty(\Omega)$ whenever $\lambda > \overline{\lambda}$ for some $\overline{\lambda}\geq 0$.

\begin{proof}[Proof of Theorem~\ref{th:General_Exist} (Case $\ell=\infty$)]

Without loss of generality, we may assume that both $\{\alpha_n\}$ and $\{\gamma_n\}$ are increasing sequences that diverge, with $\alpha_n < \gamma_n$. Moreover, we may also suppose that $F(\gamma_n) > 0$ for every $n\in\N$.

Let $\oF$ be the function defined in~\eqref{eq:def_ov_F}. Since $F\leq \oF$, passing to a subsequence if necessary, it follows from~\eqref{eq:hyp_general} that
\begin{equation}
\label{eq:Pf_large_1}
\frac{\oF(\gamma_n)}{\gamma_n^p} \to K \geq \oL > 0.
\end{equation}

Consider now the function $\uF$ defined, for all $0 \leq s_1 \leq s_2$, by
\begin{equation*}
    \uF(s_1,s_2) := \min \left\{ \int_t^{s_2} f(\rho)\, \mathrm{d}\rho : 0\leq t\leq s_1 \right\}.
\end{equation*}
Note that the identity $\uF(s_1,s_2) = F(s_2)- F(s_1) + \uF(s_1,s_1)$ holds. For each fixed $k\in\N$, it follows from~\eqref{eq:hyp_general} that
\begin{equation}
\label{eq:Pf_large_2} 
\frac{\uF(\alpha_k,\gamma_n)}{\oF(\gamma_n)} = \frac{F(\gamma_n) - F(\alpha_k) + \uF(\alpha_k,\alpha_k)}{F(\gamma_n) - \min_{t\in [0,\gamma_n]} F(t)} \geq \frac{F(\gamma_n) - F(\alpha_k) + \uF(\alpha_k,\alpha_k)}{F(\gamma_n) + M F(\gamma_n)},\quad \forall n\geq n_k,
\end{equation}
where $n_k\in\N$ is such that $\uF(\alpha_k,\gamma_n) \geq 0$ for all $n\geq n_k$; note that $F(\gamma_n)\to\infty$. Since the right-hand side of~\eqref{eq:Pf_large_2} converges to $(1+M)^{-1}$ as $n\to \infty$, a diagonal argument yields a subsequence of $\{\gamma_n\}$ (not relabelled) such that
\begin{equation}
\label{eq:Pf_large_3} 
\frac{\uF(\alpha_n, \gamma_n)}{\oF(\gamma_n)} \geq \frac{1}{2(1+M)},\quad \forall n\in \N.
\end{equation}

Regarding $\{\gamma_n\}$, we may also assume that between any two consecutive elements of $\{\alpha_n\}$ there is at most one element of $\{\gamma_n\}$. Then, there exists a subsequence $\{\alpha_{\sigma(n)}\}$ of $\{\alpha_n\}$ such that $\alpha_{\sigma(n)-1} < \gamma_n \leq \alpha_{\sigma(n)}$ for every $n$. Finally, it can also be supposed that $\alpha_n < \alpha_{\sigma(n)}$ for every $n$.

By Lemma~\ref{lem:minimizer}, for each $n\in\N$ there exists a minimizer $u_n \in \sob$ of $I_n(\lambda, \cdot)$ such that $0\leq u_n \leq\alpha_{n}$. Our aim is to show that, for each $n\in\N$, there exists $\lambda_n>0$ and $w_n\in \sob$ with $\|w_n\|_\lio\leq \alpha_{\sigma(n)}$ such that
\begin{equation}
\label{eq:Pf_large_4}
I_{\sigma(n)}(\lambda, w_n) < I_n(\lambda,u_n),\quad \forall \lambda>\lambda_n.
\end{equation}
Since $u_{\sigma(n)}$ minimizes $I_{\sigma(n)}(\lambda, \cdot)$, it follows that
\[
I_{\sigma(n)}(\lambda,u_{\sigma(n)}) \leq I_{\sigma(n)}(\lambda,w_n) < I_n(\lambda, u_n ),\quad \forall \lambda>\lambda_n,
\]
and then $\|u_{\sigma(n)} \|_{\lio} \in (\alpha_n, \alpha_{\sigma(n)}]$; otherwise, $I_{\sigma(n)} (\lambda, u_{\sigma(n)}) = I_n ( \lambda, u_{\sigma(n)} )$, contradicting the minimality of $u_n$ for $I_{n}(\lambda, \cdot)$.

To prove~\eqref{eq:Pf_large_4}, take $\delta>0$ and let $w_n\in \sob$ be as in~\eqref{eq:Pf_small_2bis}. Using the definitions of $\oF$ and $\uF$, inequality~\eqref{eq:Pf_large_3} and that $0\leq u_n \leq \alpha_n$, we deduce
\begin{equation}
\label{eq:Pf_large_5}
\begin{split}
\int_{\Omega} \left( F_{\sigma(n)} (w_n) - F_n(u_n) \right) &=
\int_{\Omega \setminus \Omega_{\delta}} F (w_n) +
\int_{\Omega_{\delta}} F(w_n) - \int_{\Omega} F (u_n)
\\
&=
\int_{\Omega} \left( F(\gamma_n) - F(u_n) \right) - \int_{\Omega_\delta} \left( F(\gamma_n)  - F(w_n) \right)
\\
&\geq \uF(\alpha_n, \gamma_n) |\Omega| - \oF(\gamma_n) |\Omega_\delta| \geq \oF(\gamma_n) \left(\frac{1}{2(1+M)} |\Omega| -|\Omega_\delta| \right),\quad \forall n\in\N.
\end{split}
\end{equation}
Now, we fix $\delta>0$ small such that $C_1:= \frac{1}{2(1+M)} |\Omega| -|\Omega_\delta| >0$. Taking into account the growth condition~(\hyperlink{hyp:growth}{H2}), the equality $\|\nabla w_n\|_{L^\infty(\Omega)} = \frac{\gamma_n}{\delta}$, and the inequality~\eqref{eq:Pf_large_5}, it follows that
\begin{align*}
I_{\sigma(n)}(\lambda,w_n) - I_n ( \lambda, u_n) &= \into \Phi(x,\nabla w_n) - \into \Phi (x,\nabla u_n ) - \lambda \int_{\Omega} \left( F_{\sigma(n)} (w_n) - F_n (u_n) \right)
\\
&\leq \into \Phi(x,\nabla w_n) - \lambda C_1 \oF(\gamma_n) \leq \frac{\beta}{p} \frac{|\Omega|}{\delta^p} \gamma_n^p - \lambda C_1 \oF(\gamma_n) =: C_2 \gamma_n^p - \lambda C_1 \oF(\gamma_n),\quad \forall n\in\N.
\end{align*}
Taking $\lambda_n := \frac{C_2}{C_1} \gamma_n^p \oF(\gamma_n)^{-1}$, we obtain~\eqref{eq:Pf_large_4}. Then, we have shown that $\|u_{\sigma(n)} \|_{\lio} \in (\alpha_n, \alpha_{\sigma(n)}]$ whenever $\lambda> \lambda_n$.

Finally, we define $\overline{\lambda} := \lim \lambda_n$, which is nonnegative and finite by~\eqref{eq:Pf_large_1}. Furthermore, it follows from~\eqref{eq:Pf_large_1} that $\overline{\lambda}=0$ if $\oL=\infty$. For each $\lambda > \overline{\lambda}$, there is some $n_0\in\N$ such that $\lambda_n<\lambda$ for any $n\geq n_0$. Hence, $\{u_{\sigma(n)}\}_{n\geq n_0}$ is a sequence of nonnegative (and nontrivial) bounded solutions to~\eqref{eq:div_problem} such that $\|u_{\sigma(n)}\|_\lio > \alpha_n$ for each $n\geq n_0$. Since $\alpha_n\to \infty$, we conclude that $\|u_{\sigma(n)}\|_\lio \to \infty$.
\end{proof}

\section{Nonexistence of bifurcation points}
\label{sec:nonexistence}

In this section, we establish results complementary to those in Section~\ref{sec:existence}. Our approach, partially inspired by~\cite{DS87}, relies on the construction of radial solutions. To handle the ordinary differential equation (ODE) satisfied by these solutions, we consider two classes of operators: the $p$-Laplacian and uniformly elliptic operators in non-divergence form.

First, we address the $p$-Laplacian operator. In what follows, we provide the proof of Theorem~\ref{th:Intro_NonExist}.

\begin{proof}[Proof of Theorem~\ref{th:Intro_NonExist}]
We carry out the proof for the cases $\ell=0$ and $\ell=\infty$ simultaneously. Without loss of generality, we may assume that $\{\alpha_n\}$ is decreasing if $\ell=0$ and increasing if $\ell=\infty$. We first establish our claim under the assumption $f(0)\geq 0$, and then we address the case $f(0)<0$. We point out that, by continuity, in the case $\ell=0$ it always holds that $f(0)=0$.
\smallskip

\textbf{Case $f(0)\geq 0$.}

Arguing by contradiction, suppose that~\eqref{eq:Pb_pLap} has a bifurcation point from $\ell$, denoted by $\tilde\lambda \geq 0$. Then there exists a sequence of pairs $(\lambda_n, u_n) \in \R^+\times \lio$, where each $u_n\geq 0$ is a nontrivial solution to~\eqref{eq:Pb_pLap}, such that $\lambda_n$ converges to $\tilde\lambda$ and $\|u_n\|_\lio\to \ell$.

We may assume that between two consecutive zeros of $f$ there is at most one element of $\{\|u_n\|_\lio\}$. Passing to a further subsequence (not relabelled), we can also assume that $\|u_n\|_\lio$ belongs to the interval $(\min\{\alpha_{n-1}, \alpha_{n+1}\}, \alpha_n]$.

Let $B_R$ be a ball of radius $R>0$ centred at $0$ containing $\Omega$. We denote by $\underline u_n$ the extension of $u_n$ by zero to $B_R\setminus\Omega$. The function $\underline u_n$ is a subsolution to
\begin{equation}
	\label{eq:PbBall}
	\begin{cases}
		-\Delta_p v = \lambda f(v) & \text{in } B_R,\\
		v=0 & \text{on } \partial B_R,
	\end{cases}
\end{equation}
with $\lambda=\lambda_n$, since $f(0)\geq 0$. The constant $\alpha_{n}$ is a supersolution. By the sub-supersolution method (see~\cite[Theorems~1.3--1.4 and Remark~1.5]{LS06}), there exists a solution $v_n$ to~\eqref{eq:PbBall} with $\lambda=\lambda_{n}$ that is maximal among the solutions contained in the interval $[0,\alpha_n]$ and satisfies $\underline u_n \leq v_n \leq \alpha_n$. Since the \mbox{$p$-Laplacian} is rotation invariant, any rotation of $v_n$ is also a solution. The maximality of $v_n$ then implies that it must be radially symmetric. Note that $v_n(0)=\|v_n\|_{L^\infty(B_R)} \in (\min\{\alpha_{n-1}, \alpha_{n+1}\}, \alpha_n]$. In particular, $v_n(0)\to \ell$ as $n \to \infty$.

Denoting $r=|x|$, the function $v_n(r)$ solves
\begin{equation}
	\label{eq:PbBall2}
	\begin{cases}
		-\left( |v_n'|^{p-2} v_n' \right)' - \frac{N-1}{r} |v_n'|^{p-2} v_n' = \lambda_{n} f(v_n), \quad r\in (0,R),\\
		v_n(R)= v_n'(0) =0.
	\end{cases}
\end{equation}
Multiplying~\eqref{eq:PbBall2} by $-v_n'$ and integrating from $0$ to $r$, one gets
\begin{equation}
\label{eq:Pf_Th_Sub_2}
\frac{p-1}{p} |v_n'(r)|^p + (N-1) \int_0^r \frac{|v_n'(t)|^p}{t} \, \mathrm{d}t = \lambda_n \int_{v_n(r)}^{v_n(0)} f(t)\, \mathrm{d}t.
\end{equation}

Equality~\eqref{eq:Pf_Th_Sub_2} with $r=R$ shows that if $v_n$ is a solution to~\eqref{eq:PbBall}, then $F(v_n(0))\geq 0$, since the left-hand side is nonnegative. Now, if $\oL<0$, then $F(s)<0$ for all $s$ in an interval $I$, where $I=(0,s_0)$ if $\ell=0$ or $I=(s_0,\infty)$ if $\ell=\infty$. Therefore, no solution to~\eqref{eq:Pb_pLap} can have its maximum in $I$, and no $\lambda\geq 0$ is a bifurcation point from $\ell$.

Now we study the case $\oL\geq 0$. Dropping the nonnegative integral term in~\eqref{eq:Pf_Th_Sub_2} and using the definition of $\oF$ (see~\eqref{eq:def_ov_F}), it follows
\[
\frac{p-1}{p} |v_n'(r)|^p \leq \lambda_n \oF(v_n(0)).
\]
By the Mean Value Theorem, there exists $r_n\in (0,R)$ with $v_n'(r_n) = -\frac{v_n(0)}{R}$. Evaluating at $r=r_n$, we obtain
\begin{equation}
\label{eq:Pf_Th_Sub_1}
\frac{p-1}{pR^p} \leq \lambda_n \frac{\oF(v_n(0))}{v_n(0)^p}.
\end{equation}
Taking into account that
\[
\limsup_{s\to \ell^+} \frac{\oF(s)}{s^p} = \limsup_{s\to \ell^+} \frac{F(s)-\min_{t\in[0,s]} F(t)}{s^p} \leq \limsup_{s\to \ell^+} \frac{F(s)}{s^p} - \liminf_{s\to \ell^+} \frac{\min_{t\in[0,s]} F(t)}{s^p} = \oL - \min\{0,\uL\},
\]
we can pass to the limit in~\eqref{eq:Pf_Th_Sub_1} to deduce that
\begin{equation*}
\frac{p-1}{pR^p} \leq \tilde\lambda \limsup_{n\to\infty} \frac{\oF(v_n(0))}{v_n(0)^p} \leq \tilde\lambda (\oL - \min\{0,\uL\}).
\end{equation*}
Since $\tilde\lambda$ is an arbitrary bifurcation point, we conclude that no bifurcation point from $\ell$ exists for~\eqref{eq:Pb_pLap} in the interval $[0,\underline\lambda)$, where 
\[
\underline{\lambda} :=  \begin{cases}
\frac{p-1}{pR^p (\oL - \min\{0,\uL\})} & \text{if } \oL > 0 \text{ or } \uL < \oL, \\
\infty & \text{if } \uL=\oL = 0.
\end{cases}
\]

\textbf{Case $f(0)<0$.}

As previously noted, this case can only occur when $\ell=\infty$, where $\{\alpha_n\}$ is increasing.
Define $g\in C(\R)$ as $g(s) = f^+(s)$ for $s\in [0,\alpha_1]$, and $g(s) = f(s)$ for $s\geq \alpha_1$. Define $G$ analogously to $F$. It holds that
\[
\liminf_{s\to \infty} \frac{G(s)}{s^p} = \liminf_{s\to \infty} \frac{F(s)}{s^p} \qquad \text{and} \qquad \limsup_{s\to \infty} \frac{G(s)}{s^p} = \limsup_{s\to \infty} \frac{F(s)}{s^p}.
\]
Any nonnegative solution to~\eqref{eq:Pb_pLap} is a subsolution to $-\Delta_p v = \lambda g(v)$, and the constants $\alpha_n$ are supersolutions. By the sub-supersolution method, any bifurcation point from infinity for~\eqref{eq:Pb_pLap} is also one for the $g$-problem. Applying the previous case to the $g$-problem, we obtain our claim for~\eqref{eq:Pb_pLap}.
\end{proof}

\begin{remark} 
Combining equality~\eqref{eq:Pf_Th_Sub_2} with the arguments in~\cite[Section~3]{DS87}, we can deduce that any bounded solution $u\geq 0$ to~\eqref{eq:Pb_pLap} must satisfy the area condition
\[
\int_s^{\|u\|_\lio} f(t)\, \mathrm{d}t >0,\ \forall s\in [0,\|u\|_\lio).
\]
This phenomenon was first observed in~\cite{CS87, DS87} for the Laplacian operator.
\end{remark}

We now establish a similar result for the operators in non-divergence form defined in~\eqref{eq:nondiv_op}. As we show, condition~\eqref{eq:hyp_infsup2} is replaced by a generalized version involving the ellipticity constant $\Lambda$. Moreover, to apply the symmetrization results of~\cite{DS07}, we require $f$ to be locally Lipschitz.

\begin{theorem}
Let $\ell\in\{0,\infty\}$, and let $f$ be a locally Lipschitz function. Assume that $\mathcal{L}$ satisfies 
\eqref{eq:nondiv_op}. Suppose that~\eqref{eq:hyp_zeros} holds and that $F_\Lambda$, as defined in~\eqref{eq:def_F_Lambda}, satisfies
\begin{equation}
\label{eq:hyp_infsup_nondiv}
    -\infty < \uLL = \liminf_{s\to \ell^+} \frac{F_\Lambda(s)}{s^2} 
    \leq \limsup_{s\to \ell^+} \frac{F_\Lambda(s)}{s^2} = \oLL < \infty.
\end{equation}
Then, there exists $\underline{\lambda}\in (0,\infty]$ such that 
no $\lambda\in [0, \underline{\lambda})$ is a bifurcation point 
from $\ell$ for problem~\eqref{eq:nondiv_problem}.

Moreover, if $\oLL < 0$ or $\uLL = \oLL = 0$, then $\underline{\lambda}=\infty$.
\end{theorem}

\begin{remark}
Since $F(s) \leq F_\Lambda(s) \leq \int_0^s f^+(\rho)\,\mathrm{d}\rho$ 
for all $s\geq 0$, condition~\eqref{eq:hyp_infsup_nondiv} is implied by 
the $\Lambda$-independent conditions
\[
-\infty < \liminf_{s\to \ell^+} \frac{F(s)}{s^2} 
\qquad\text{and}\qquad 
\limsup_{s\to \ell^+} \frac{\int_0^s f^+(\rho)\,\mathrm{d}\rho}{s^2} < \infty.
\]
Note also that~\eqref{eq:hyp_infsup_nondiv} coincides with~\eqref{eq:hyp_infsup2} when $\Lambda=1$.
\end{remark}

\begin{proof}
The general strategy follows the same lines as the proof of Theorem~\ref{th:Intro_NonExist}.
Again, we carry out the proof for the cases $\ell=0$ and $\ell=\infty$ simultaneously. We may assume that $\{\alpha_n\}$ is decreasing if $\ell=0$ and increasing if $\ell=\infty$. We first establish our claim under the assumption $f(0)\geq 0$, and then we address the case $f(0)<0$.
\smallskip

\textbf{Case $f(0)\geq 0$.}

Arguing by contradiction, suppose that~\eqref{eq:nondiv_problem} has a bifurcation point from $\ell$, denoted by $\tilde\lambda \geq 0$. Then there exists a sequence of pairs $(\lambda_n, u_n) \in \R^+\times \lio$, where each $u_n\geq 0$ is a nontrivial solution to~\eqref{eq:nondiv_problem}, such that $\lambda_n$ converges to $\tilde\lambda$ and $\|u_n\|_\lio\to \ell$. We may assume that $\|u_n\|_\lio\in (\min\{\alpha_{n-1}, \alpha_{n+1}\}, \alpha_n]$ for each $n\in\N$.

Let $B_R$ be a ball of radius $R>0$ centred at $0$ containing $\Omega$. Extending $u_n$ by zero to $B_R\setminus\Omega$, and noting that $f(0) \geq 0$, the resulting extension $\underline{u}_n$ is a subsolution to
\begin{equation*}
	\begin{cases}
		-\mathcal{M}^+_{\Lambda}(D^2v) = \lambda_n f(v) & \text{in } B_R,\\
		v=0 & \text{on } \partial B_R.
	\end{cases}
\end{equation*}
Using $\alpha_{n}$ as supersolution, \cite[Theorem 2.3.1]{Sat73} yields a solution $v_n$ satisfying $\underline{u}_n \leq v_n \leq \alpha_n$. By \cite{DS07}, $v_n$ is radial with $v_n'(r) < 0$ for $r \in (0, R)$. Note that $v_n(0)=\|v_n\|_{L^\infty(B_R)} \in (\min\{\alpha_{n-1}, \alpha_{n+1}\}, \alpha_n]$. In particular, $v_n(0)\to \ell$ as $n \to \infty$.

Since $v_n'(r) < 0$, the eigenvalues of $D^2v_n$ are $v_n''(r)$ and $\frac{v_n'(r)}{r} < 0$ (with multiplicity $N-1$), so the equation becomes
\begin{equation}\label{eq:radial_var}
-\alpha(r) v_n''(r) - \frac{1}{\Lambda}(N-1)\frac{v_n'(r)}{r} = \lambda_n f(v_n(r)),
\end{equation}
where $\alpha(r) = \frac{1}{\Lambda}$ if $v_n''(r) \leq 0$ and $\alpha(r) = \Lambda$ if $v_n''(r) \geq 0$.

Multiplying~\eqref{eq:radial_var} by $-v_n'(r)$ and integrating from $0$ to $r$, it follows
\begin{equation}\label{eq:integrated}
\int_0^r \alpha(t) v_n''(t) v_n'(t) \, \mathrm{d}t + \frac{N-1}{\Lambda}\int_0^r \frac{(v_n'(t))^2}{t} \, \mathrm{d}t = \lambda_n \int_0^r f(v_n(t))(-v_n'(t)) \, \mathrm{d}t.
\end{equation}
We decompose the first term using $v_n'(0) = 0$ as
\begin{equation}
\label{eq:integrated2}
\int_0^r \alpha(t) v_n''(t) v_n'(t) \, \mathrm{d}t = \frac{1}{2\Lambda}(v_n'(r))^2 + \left(\Lambda - \frac{1}{\Lambda}\right) \int_{A_n^+} v_n''(t) v_n'(t) \, \mathrm{d}t,
\end{equation}
where $A_n^+ = \{t \in (0,r) : v_n''(t) \geq 0\}$. Similarly, let $A_n^- = \{t \in (0,r) : v_n''(t) \leq 0\}$. In the region $A_n^+$, equation~\eqref{eq:radial_var} becomes
\[
v_n''(t) = -\frac{\lambda_n}{\Lambda} f(v_n(t)) - \frac{N-1}{\Lambda^2}\frac{v_n'(t)}{t},
\]
so, multiplying by $\left(\Lambda - \frac{1}{\Lambda}\right) v_n'(t)$ and integrating, we obtain
\begin{equation}
\label{eq:integrated3}
\left(\Lambda - \frac{1}{\Lambda}\right)\int_{A_n^+} v_n''(t) v_n'(t) \, \mathrm{d}t = \left(1- \frac{1}{\Lambda^2}\right) \lambda_n \int_{A_n^+} f(v_n)(-v_n') \, \mathrm{d}t - \left(1- \frac{1}{\Lambda^2}\right) \frac{N-1}{\Lambda}\int_{A_n^+} \frac{(v_n')^2}{t} \, \mathrm{d}t.
\end{equation}
Combining~\eqref{eq:integrated},~\eqref{eq:integrated2} and~\eqref{eq:integrated3}, and simplifying the resulting expression, we deduce
\begin{align*}
\frac{1}{2\Lambda}(v_n'(r))^2 &+ \frac{N-1}{\Lambda}\int_{A_n^-}\frac{(v_n')^2}{t}\,\mathrm{d}t + \frac{N-1}{\Lambda^3}\int_{A_n^+}\frac{(v_n')^2}{t}\,\mathrm{d}t = \lambda_n\int_{A_n^-} f(v_n)(-v_n')\,\mathrm{d}t + \frac{\lambda_n}{\Lambda^2}\int_{A_n^+} f(v_n)(-v_n')\,\mathrm{d}t.
\end{align*}
Dropping the nonnegative terms on the left-hand side and performing the change of variables $s = v_n(t)$, with $E_n^+ = v_n(A_n^+)$ and $E_n^- = v_n(A_n^-)$, we obtain the estimate
\begin{align}
\frac{1}{2\Lambda}(v_n'(r))^2 &\leq \lambda_n\left(\int_{E_n^-} f(s) \, \mathrm{d}s + \frac{1}{\Lambda^2}\int_{E_n^+} f(s) \, \mathrm{d}s\right). \label{eq:key_ineq_pre}
\end{align}

By splitting $f$ into its positive and negative parts, $f=f^+-f^-$, and bounding the right-hand side of~\eqref{eq:key_ineq_pre} using that $E_n^- \cup E_n^+ = [v_n(r), v_n(0)]$, it follows
\begin{equation}\label{eq:key_ineq_3}
\frac{1}{2\Lambda}(v_n'(r))^2 \leq \lambda_n \left( \int_{v_n(r)}^{v_n(0)} f^+(s) \, \mathrm{d}s - \frac{1}{\Lambda^2} \int_{v_n(r)}^{v_n(0)} f^-(s) \, \mathrm{d}s \right) = \lambda_n \big( F_\Lambda(v_n(0))-F_\Lambda(v_n(r)) \big).
\end{equation}

In the case where $\oLL<0$, by setting $r=R$ in~\eqref{eq:key_ineq_3}, we can follow the same reasoning as in Theorem~\ref{th:Intro_NonExist} to conclude that no $\lambda\geq 0$ is a bifurcation point from $\ell$.

Now we study the case $\oLL\geq 0$. Using the function $\oF_\Lambda$ defined in~\eqref{eq:def_ov_F_Lambda}, we deduce from~\eqref{eq:key_ineq_3} that
\begin{equation}\label{eq:key_ineq}
\frac{1}{2\Lambda}(v_n'(r))^2 \leq \lambda_n\, \overline{F}_\Lambda(v_n(0)).
\end{equation}
By the Mean Value Theorem, there exists $r_n \in (0,R)$ with $|v_n'(r_n)| = \frac{v_n(0)}{R}$. Evaluating~\eqref{eq:key_ineq} at $r = r_n$, we obtain
\begin{equation}\label{eq:key_ineq_2}
\frac{1}{2\Lambda R^2} \leq \lambda_n \frac{\overline{F}_\Lambda(v_n(0))}{v_n(0)^2}.
\end{equation}
Since $\overline{F}_\Lambda(s) = F_\Lambda(s) - \min_{t\in[0,s]} F_\Lambda(t)$, it holds that
\[
\limsup_{s\to \ell^+} \frac{\overline{F}_\Lambda(s)}{s^2} \leq \oLL - \min \left\{0,\uLL \right\}.
\]
Hence, passing to the limit in~\eqref{eq:key_ineq_2}, we deduce
\[
\frac{1}{2\Lambda R^2} \leq \tilde\lambda\, \left(\oLL - \min\left\{0,\uLL\right\} \right).
\]
Since $\tilde\lambda$ is an arbitrary bifurcation point, we conclude that no bifurcation point from $\ell$ exists in $[0,\underline\lambda)$, where
\[
\underline{\lambda} := \begin{cases}
\dfrac{1}{2\Lambda R^2 \left(\oLL - \min\left\{0,\uLL\right\} \right)} & \text{if } \oLL > 0 \text{ or } \uLL < \oLL,  \\[3mm]
\infty & \text{if } \uLL = \oLL = 0.
\end{cases}
\]

\textbf{Case $f(0)<0$.}

By continuity, this case can only occur when $\ell=\infty$, where $\{\alpha_n\}$ is increasing. Let us define the auxiliary function $g\in C(\R)$ as $g(s) = f^+(s)$ for $s\in [0,\alpha_1]$, and $g(s) = f(s)$ for $s\geq \alpha_1$. Furthermore, let $G_\Lambda$ be defined analogously to $F_\Lambda$ (see~\eqref{eq:def_F_Lambda}). Since $g$ and $f$ coincide on $[\alpha_1,\infty)$, the functions $G_\Lambda$ and $F_\Lambda$ differ by a constant for $s\geq \alpha_1$, so that
\[
\liminf_{s\to \infty} \frac{G_\Lambda(s)}{s^2} = \liminf_{s\to \infty} \frac{F_\Lambda(s)}{s^2} \qquad \text{and} \qquad \limsup_{s\to \infty} \frac{G_\Lambda(s)}{s^2} = \limsup_{s\to \infty} \frac{F_\Lambda(s)}{s^2}.
\]
Any nonnegative solution to~\eqref{eq:nondiv_problem} is a subsolution to $-\mathcal{M}^+_\Lambda(D^2 v) = \lambda g(v)$, and the constants $\alpha_n$ are supersolutions. By the sub-supersolution method, any bifurcation point from infinity for~\eqref{eq:nondiv_problem} is also one for problem
\begin{equation*}
	\begin{cases}
		-\mathcal{M}^+_{\Lambda}(D^2u) = \lambda g(u) & \text{in } \Omega,\\
		u=0 & \text{on } \partial \Omega.
	\end{cases}
\end{equation*}
Applying the previous case to this problem, we obtain our claim for~\eqref{eq:nondiv_problem}.
\end{proof}

\section*{Acknowledgements}

The first author is supported by
the FPU predoctoral fellowship of the Spanish Ministry of Universities (FPU21/04849). The second author has received funding from the European Research Council (ERC) under the Grant Agreement No.\ 862342, from AEI project PID2024-156429NB-I00 (Spain), and from the Grant CEX2023-001347-S funded by MICIU/AEI/10.13039/501100011033 (Spain).

\end{document}